\documentclass{amsart}
\usepackage{graphicx}
\usepackage{amssymb}

\newtheorem{thm}{Theorem}[section]

\newtheorem{prop}[thm]{Proposition}
\theoremstyle{definition}

\theoremstyle{remark}
\newtheorem{rem}[thm]{Remark}

\begin{document}

\title[Strict $2$-convexity]{Strict $2$-convexity of convex solutions to the quadratic Hessian equation}
\author{Connor Mooney}
\address{Department of Mathematics, UC Irvine}
\email{\tt mooneycr@math.uci.edu}

\begin{abstract}
We prove that convex viscosity solutions to the quadratic Hessian inequality
$$\sigma_2(D^2u) \geq 1$$
are strictly $2$-convex. As a consequence we obtain short proofs of smoothness and interior $C^2$ estimates for convex viscosity solutions to $\sigma_2(D^2u) = 1$, which were proven using different methods in recent works of Guan-Qiu \cite{GQ}, McGonagle-Song-Yuan \cite{MSY} and Shankar-Yuan \cite{SY2}.
\end{abstract}
\maketitle

\section{Introduction}
In this note we consider convex viscosity solutions to the quadratic Hessian inequality
\begin{equation}\label{QuadHessIneq}
\sigma_2(D^2u) \geq 1.
\end{equation}
Our main result is their strict two-convexity. That is:
\begin{thm}\label{Main}
Let $u$ be a convex viscosity solution to (\ref{QuadHessIneq}) in $\Omega \subset \mathbb{R}^n$, and let $L$ be a supporting linear function to $u$ in $\Omega$. Then
$$\text{dim}\{u = L\} \leq n-2.$$
\end{thm}
\noindent Theorem \ref{Main} is sharp in view of the example $u = x_1^2 + x_2^2,$ with $L = 0.$

Local smoothness of convex viscosity solutions to
\begin{equation}\label{QuadHessEq}
\sigma_2(D^2u) = 1
\end{equation}
follows from Theorem \ref{Main}, using the classical solvability of the Dirichlet problem \cite{CNS} and the Pogorelov-type interior $C^2$ estimate from \cite{CW} (see Section \ref{Preliminaries}). With a compactness argument we can in fact prove a universal modulus of strict $2$-convexity (see Proposition \ref{Universal}). As a result we obtain:
\begin{thm}\label{Estimate}
Let $u$ be a convex viscosity solution of (\ref{QuadHessEq}) in $B_1 \subset \mathbb{R}^n$. Then $u$ is smooth, and
\begin{equation}\label{C2Est}
|D^2u(0)| \leq C\left(n,\|u\|_{L^{\infty}(B_1)}\right).
\end{equation}
\end{thm}
Inequality (\ref{C2Est}) was recently proven for smooth convex solutions of (\ref{QuadHessEq}) in \cite{GQ} and \cite{MSY}, and Theorem \ref{Estimate} was proven in \cite{SY2}. A subtle issue in passing to the viscosity case is that smooth approximations of convex viscosity solutions may not be convex. An advantage of our approach is that it avoids using a priori estimates for smooth convex solutions, which allows us to bypass this issue. The methods in the above-mentioned works are quite different from ours, based in \cite{GQ} on the Bernstein technique, and in \cite{MSY} and \cite{SY2} on the properties of the equation for the Legendre-Lewy transform of $u$.

An interesting question is whether the conclusion of Theorem \ref{Estimate} holds without assuming that $u$ is convex. It is true when $n = 2$ (in which case solutions are automatically convex and (\ref{QuadHessEq}) is the Monge-Amp\`{e}re equation, \cite{H}) and when $n = 3$ (in which case (\ref{QuadHessEq}) is equivalent to the special Lagrangian equation, \cite{WY}). It is also known to be true if $u$ is slightly non-convex \cite{SY2}. Finally, an interior $C^2$ estimate of the form (\ref{C2Est}) was recently obtained in \cite{SY1} for smooth solutions to (\ref{QuadHessEq}) that satisfy the semi-convexity condition $D^2u \geq -KI$, with $C$ depending also on $K$. The general case in dimension $n \geq 4$ remains open.

\begin{rem}
Local smoothness and interior $C^2$ estimates are false for convex viscosity solutions to the $k$-Hessian equation 
$$\sigma_k(D^2u) = 1$$ 
when $k \geq 3$, in view of the well-known Pogorelov example (\cite{P}, \cite{U}). The same example shows that convex viscosity solutions to $\sigma_k(D^2u) \geq 1$ are not always strictly $k$-convex when $k \geq 3$. In particular, Theorems \ref{Main} and \ref{Estimate} are both special to the quadratic Hessian equation.
\end{rem}

The paper is organized as follows. In Section \ref{Preliminaries} we recall a few classical results about the $k$-Hessian equation, and we use them to show that Theorem \ref{Main} implies that convex viscosity solutions of (\ref{QuadHessEq}) are smooth. In Section \ref{ProofMain} we prove Theorem \ref{Main}. Finally, in Section \ref{ProofEstimate} we prove a quantitative version of Theorem \ref{Main} using a compactness argument, and we use it to complete the proof of Theorem \ref{Estimate}.

\section*{Acknowledgments}
The author is grateful to Ravi Shankar and Yu Yuan for comments. This research was supported by NSF grant DMS-1854788.

\section{Preliminaries}\label{Preliminaries}
In this section we recall a few classical facts about the $k$-Hessian equation. Below $\Omega$ denotes a bounded domain in $\mathbb{R}^n$, and
$1 \leq k \leq n$.

We first recall some facts about the $\sigma_k$ operator. The function $\sigma_k$ on $Sym_{n \times n}$ denotes the $k^{th}$ symmetric polynomial of the eigenvalues. It is elliptic on the cone
$$\Gamma_k := \{M \in Sym_{n \times n}: \sigma_l(M) > 0 \text{ for each } 1 \leq l \leq k\},$$
and has convex level sets in $\Gamma_k$. Furthermore, the function $\sigma_k$ is uniformly elliptic on compact subsets of $\Gamma_k$.

Next we recall the notion of viscosity solution. We say that a function $u \in C^2(\Omega)$ is $k$-convex if $D^2u \in \overline{\Gamma_k}$. Given a nonnegative function $f \in C(\Omega)$, we say that a function $u \in C(\Omega)$ is a viscosity solution of
$$\sigma_k(D^2u) \geq (\leq) \, f$$
if, whenever a $k$-convex function $\varphi \in C^2(\Omega)$ touches $u$ from above (below) at a point $x_0 \in \Omega$, we have 
$$\sigma_k(D^2\varphi(x_0)) \geq (\leq) \, f(x_0).$$ 
We say that $u \in C(\Omega)$ is a viscosity solution of 
$$\sigma_k(D^2u) = f$$ 
if it is a viscosity solution of both $\sigma_k(D^2u) \geq f$ and $\sigma_k(D^2u) \leq f$. The notions of classical and viscosity solution coincide on $C^2$ $k$-convex functions.

Third we recall the classical solvability of the Dirichlet problem for the $k$-Hessian equation, proven in \cite{CNS}:
\begin{thm}\label{DP}
Let $g \in C^{\infty}(\partial B_R)$. Then there exists a unique $k$-convex solution $u \in C^{\infty}\left(\overline{B_R}\right)$ to the Dirichlet problem
$$\sigma_k(D^2u) = 1 \text{ in } B_R, \quad u|_{\partial B_R} = g.$$
\end{thm}
\noindent The result in fact holds for smooth bounded $k-1$-convex domains.

Finally we recall the Pogorelov-type estimate Theorem $4.1$ from \cite{CW}:
\begin{thm}\label{Pogo}
Assume that $u \in C^{\infty}\left(\overline{\Omega}\right)$ is a $k$-convex solution to
$$\sigma_k(D^2u) = 1 \text{ in } \Omega,$$
and that there exists a $k$-convex function $w \in C\left(\overline{\Omega}\right)$ such that $u < w$ in $\Omega$ and $u = w$ on $\partial \Omega$. Then
\begin{equation}\label{PogoEst}
\sup_{\Omega} \left((w-u)^4|D^2u|\right) \leq C\left(n,\,k,\, \|u\|_{C^1(\Omega)}\right).
\end{equation}
\end{thm}
\noindent Inequality (\ref{PogoEst}) implies in particular that the equation for $u$ is uniformly elliptic on compact subdomains of $\Omega$. By the Evans-Krylov theorem (see \cite{CC}), interior derivative estimates of all higher orders follow.

To conclude the section we show local smoothness of convex viscosity solutions to (\ref{QuadHessEq}). We assume $u$ is defined in $B_1 \subset \mathbb{R}^n$ , and it suffices to prove smoothness in a neighborhood of the origin. After subtracting a supporting linear function we may assume that $u(0) = 0$ and that $u \geq 0$. By Theorem \ref{Main} we have after a rotation that $\{u = 0\}$ is contained in 
the subspace spanned by $\{e_3,\,...,\,e_n\}$. Let
$$w_{\delta}(x) := \delta[2(n-2)(x_1^2 + x_2^2) - (x_3^2 + ... + x_n^2)],$$
and notice that $w_{\delta}$ is $2$-convex for all $\delta > 0$. Furthermore, we can choose $\delta,\, \eta,\, \mu > 0$ small (depending on $u$) such that
$$u > w_{\delta} + \eta \text{ on } \partial B_{1/2} \quad \text{ and } \quad \overline{B_{\mu}} \subset \{u < w_{\delta} + \eta\}.$$
Let $\{v_j\}$ be a sequence of smooth $2$-convex (but not necessarily convex) solutions to (\ref{QuadHessEq}) that converge
uniformly to $u$ in $B_{1/2}$. (One obtains the functions $v_j$ e.g. by taking smooth approximations to $u$ on $\partial B_{1/2}$ and applying Theorem \ref{DP} with $R = 1/2$ and $k = 2$.) Applying Theorem \ref{Pogo} to $v_j$ with $w = w_{\delta} + \eta$ and $k = 2$, we see that the solutions $v_j$ enjoy uniform derivative estimates of all orders in $B_{\mu}$ as $j \rightarrow \infty$. We conclude that $u$ is smooth in $B_{\mu}$.

\section{Proof of Theorem \ref{Main}}\label{ProofMain}

In this section we prove Theorem \ref{Main}.

\begin{proof}[{\bf Proof of Theorem \ref{Main}:}]
Assume by way of contradiction that there exists a supporting linear function $L$ to $u$ such that $\text{dim}\{u = L\} \geq n-1.$
After subtracting $L$, translating, rotating, and quadratically rescaling, we may assume that $u$ is defined in $B_2$, that $u \geq 0$, and that $u = 0$ on $\{x_n = 0\} \cap B_2$. After subtracting another supporting linear function of the form $a x_n$ with $a \geq 0$, we may also assume that
$$u(te_n) = o(t) \text{ as } t \rightarrow 0^+.$$
Letting $x = (x',\,x_n)$, it follows that $\{u < h\}$ contains a cylinder of the form
$$Q_h := \{|x'| < 1\} \times (0,\,H),$$
with $h/H \rightarrow 0$ as $h \rightarrow 0^+.$ For $h$ small, the convex paraboloid 
$$P_h := h|x'|^2 + 4\frac{h}{H^2}(x_n - H/2)^2$$
thus satisfies that $P_h \geq h \geq u$ on $\partial Q_h$, that $P_h(He_n/2) = 0 \leq u,$ and that
$$\sigma_2(D^2P_h) = c_1(n)h^2 + c_2(n)\frac{h^2}{H^2} < 1,$$
which contradicts (\ref{QuadHessIneq}).
\end{proof}

\section{Proof of Theorem \ref{Estimate}}\label{ProofEstimate}
In this section we prove a quantitative version of Theorem \ref{Main}, and we use it to complete the proof of Theorem \ref{Estimate}. 
For a set $S \subset \mathbb{R}^n$ and $r > 0$ we let $S_r$ denote the $r$-neighborhood of $S$.

\begin{prop}\label{Universal}
For $K > 0,\,r > 0$ and $n \geq 2$, there exists $\delta(n,\,K,\,r) > 0$ such that if $u$ is a convex viscosity solution to (\ref{QuadHessIneq}) in $B_1 \subset \mathbb{R}^n$ with $\|u\|_{L^{\infty}(B_1)} \leq K$ and $L$ is a supporting linear function to $u$ at $0$, then
$$\{u < L + \delta\} \subset \subset T_r$$
for some $n-2$-dimensional subspace $T$ of $\mathbb{R}^n$.
\end{prop}

\begin{proof}
Assume not. Then there exist convex viscosity solutions $u_j$ to (\ref{QuadHessIneq}) on $B_1$ with $\|u_j\|_{L^{\infty}(B_1)} \leq K$ and supporting linear functions $L_j$ at $0$ such that the conclusion fails with $\delta = 1/j$. Up to taking a subsequence, the functions $u_j$ converge locally uniformly to a convex viscosity solution $v$ of $(\ref{QuadHessIneq})$ in $B_1$, and $L_j$ converge to a supporting linear $L$ to $v$ at $0$ such that $\{v = L\}$ is not compactly contained in $T_r$ for any $n-2$-dimensional subspace $T$. This contradicts Theorem \ref{Main}. 
\end{proof}

\begin{proof}[{\bf Proof of Theorem \ref{Estimate}:}]
We proved that $u$ is smooth at the end of Section \ref{Preliminaries}. The proof of the estimate (\ref{C2Est}) follows the same lines.
We call a constant universal if it depends only on $n$ and $\|u\|_{L^{\infty}(B_1)}$. We may assume after subtracting a linear function with universal $C^1$ norm that $u(0) = 0$ and that $u \geq 0$. Write $x = (y,\,z)$ with $y \in \mathbb{R}^2$ and $z \in \mathbb{R}^{n-2}$. By Proposition \ref{Universal} there exists $\delta > 0$ universal such that, after a rotation, $u > \delta$ on $\{|y| = 1/(2n)\} \cap B_1$. It follows that
$$u > w := \delta\left(2(n-2)|y|^2 - |z|^2 + \frac{1}{8}\right)$$
on the boundary of $B_{3/4} \cap \{|y| < 1/(2n)\}$. Notice also that $w$ is $2$-convex. The estimate (\ref{C2Est}) follows by applying Theorem \ref{Pogo} in the connected component of the set $\{u < w\}$ that contains the origin.
\end{proof}




\begin{thebibliography}{9999}
\bibitem[CC]{CC} Caffarelli, L.; Cabr\'{e}, X. {\it Fully Nonlinear Elliptic Equations}. Colloquium Publications 43. Providence, RI: American Mathematical Society, 1995.
\bibitem[CNS]{CNS} Caffarelli, L.; Nirenberg, L.; J. Spruck, J. The Dirichlet problem for nonlinear second order elliptic equations, III: Functions of the eigenvalues of the Hessian. {\it Acta Math.} {\bf 155} (1985), 261-301.
\bibitem[CW]{CW} Chou, K.-S.; Wang, X.-J. A variational theory of the Hessian equation. {\it Comm. Pure Appl. Math.} {\bf 54} (2001), 1029-1064.
\bibitem[GQ]{GQ} Guan, P.; Qiu, G. Interior $C^2$ regularity of convex solutions to prescribing scalar curvature equations. {\it Duke Math. J.} {\bf 168} (2019), no. 9, 1641-1663.
\bibitem[H]{H} Heinz, E. On elliptic Monge-Amp\`{e}re equations and Weyl's embedding problem. {\it J. Analyse Math.} {\bf 7} (1959),1-52.
\bibitem[MSY]{MSY} McGonagle, M.; Song, C.; Yuan, Y. Hessian estimates for convex solutions to quadratic Hessian equation. {\it Ann. Inst. H. Poincar\'{e} Anal. Non Lin\'{e}aire} {\bf 36} (2019), no. 2, 451-454.
\bibitem[P]{P} Pogorelov, A. V. {\it The multidimensional Minkowski problem.} Izdat. ``Nauka," Moscow, 1975.
\bibitem[SY1]{SY1} Shankar, R.; Yuan, Y. Hessian estimate for semiconvex solutions to the sigma-2 equation. {\it Calc. Var. Partial Differential Equations}, to appear.
\bibitem[SY2]{SY2} Shankar, R.; Yuan, Y. Regularity for almost convex viscosity solutions of the sigma-2 equation. {\it J. Math. Study}, to appear.
\bibitem[U]{U} Urbas, J. On the existence of nonclassical solutions for two classes of fully nonlinear elliptic equations. {\it Indiana Univ. Math. J.} {\bf 39} (1990), no. 2, 355-382.
\bibitem[WY]{WY} Warren, M.; Yuan, Y. Hessian estimates for the sigma-2 equation in dimension three. {\it Comm. Pure Appl. Math.} {\bf 62} (2009), no. 3, 305-321.
\end{thebibliography}
\end{document}